\documentclass[reqno]{amsart}
\usepackage{amsthm}
\usepackage{amssymb}
\usepackage{algorithmic} 
\usepackage{algorithm} 
\usepackage{graphics}
\usepackage{graphicx}
\usepackage{tikz}
\usepackage{hyperref}
\usepackage{caption}
\usepackage{mathabx}
\usepackage{cleveref}
\usepackage{hyperref}
\usepackage{mathtools}
\usepackage{todonotes}

\newtheorem{proposition}{Proposition}[section]
\newtheorem{theorem}[proposition]{Theorem}
\newtheorem{corollary}[proposition]{Corollary}
\newtheorem{lemma}[proposition]{Lemma}
\newtheorem{definition}[proposition]{Definition}
\newtheorem{assumption}[proposition]{Assumption}
\newtheorem{example}[proposition]{Example}

\newcommand{\aff}{\mathrm{aff}\,}
\def\argmin{ \mathop{{\rm argmin}}}

\newcommand{\diag}{\mathrm{diag}}
\newcommand{\Diag}{\mathrm{DIAG}}
\newcommand{\dom}{\mathrm{dom}\,}
\newcommand{\epi}{\mathrm{epi}\,}

\newcommand{\inter}{\mathrm{int}\,}

\newcommand{\ri}{\mathrm{ri}\,}

\newcommand{\rge}{\mathrm{rge}\,}
\newcommand{\para}{\mathrm{par}\,}

\newcommand{\sgn}{\mathrm{sgn}}

\newcommand{\lin}{\mathrm{span}\,}

\newcommand{\p}{\partial}

\newcommand{\R}{\mathbb{R}}

\newcommand{\bS}{\mathbb{S}}

\newcommand{\rank}{\mathrm{rank}\,}

\newcommand{\rp}{\mathbb R\cup\{+\infty\}}
\newcommand{\bR}{\mathbb{R}}

\def\bS{\mathbb{S}}

\newcommand{\IFF}{\quad\Longleftrightarrow\quad}

\newcommand{\tr}{\mathrm{tr}\,}

\newcommand{\bA}{\mathbb{A}}
\newcommand{\bB}{\mathbb{B}}

\newcommand{\bE}{\mathbb{E }}

\newcommand{\ip}[2]{\left\langle #1,\, #2\right\rangle}
\newcommand{\half}{\frac{1}{2}}

\newcommand{\set}[2]{\left\{#1\,\left\vert\; #2\right.\right\}}

\newcommand{\sig}{\sigma}

\newcommand{\cA}{\mathcal{A}}
\newcommand{\cC}{\mathcal{C}}

\newcommand{\cL}{\mathcal{L}}

\newcommand{\cS}{\mathcal{S}}
\newcommand{\cT}{\mathcal{T}}

\newcommand{\cX}{\mathcal{X}}
\newcommand{\cV}{\mathcal{V}}

\newcommand{\AND}{\ \mbox{ and }\ }
\newcommand{\st}{\ \mbox{s.t.}\ }

\date{\today}

\author{Tim Hoheisel}
\address{Department of Mathematics and Statistics, McGill University, 805 Sherbrooke St West, Montr\'eal, Qu\'ebec, Canada H3A 0B9}
\email{tim.hoheisel@mcgill.ca}
\thanks{The first  and second author are partially supported by an NSERC discovery grant}

\author{Elliot Paquette}
\address{Department of Mathematics and Statistics, McGill University, 805 Sherbrooke St West,  Montr\'eal, Qu\'ebec, Canada H3A 0B9}
\email{elliot.paquette@mcgill.ca}

\begin{document}

\title[Uniqueness in nuclear norm minimization]{Flatness of the nuclear norm sphere, simultaneous polarization, and uniqueness in nuclear norm minimization}



\date{\today}


\begin{abstract} In this paper we establish necessary and sufficient conditions for the existence  of  line
segments (or {\em flats}) in the  sphere of the nuclear norm via the notion of {\em simultaneous polarization} and a refined expression for the subdifferential of the nuclear norm.  This is then leveraged to provide (point-based) necessary and sufficient conditions for uniqueness of solutions  for minimizing the   nuclear norm  over an affine manifold. We further  establish an alternative  set of sufficient conditions for uniqueness, based on the interplay of the subdifferential of the nuclear norm and the range of the problem-defining linear operator. Finally, using convex duality, we show how to transfer the uniqueness results for the original problem to a whole class of nuclear norm-regularized minimization problems with a strictly convex fidelity term.
\end{abstract}

\maketitle

\section{Introduction}

\noindent
One of the most ubiquitous paradigms for   linear inverse problems in matrix space  is  {\em low rank approximation},  often cast in the form 
\begin{equation}\label{eq:LR} 
\min_{X\in \R^{n\times p}} \rank X\st \cA(X)=b.
\end{equation}
Here  $\cA:\R^{n\times p}\to \bE$ is a linear map (into a Euclidean space $\bE$) whose action is  often simply a matrix  multiplication $\cA(X)=A\cdot X$ for some $A\in \R^{m\times n}$ or a selection operator which projects $X$ onto the matrix composed of its entries from a prescribed index set $J\subset \{1,\dots, n\} \times \{1,\dots,p\}$. We direct the interested reader   to  Fazel's thesis \cite{Faz 02}, the important paper by Cand\`es and Recht \cite{CaR 09} as well as  the survey article by  Recht et al. \cite{RFP 10} for applications, solution methods and pointers to the abundant literature  for the low rank minimization problem  \eqref{eq:LR} and the low rank minimization paradigm in general.

Due to the combinatorial nature  of the rank function, problem \eqref{eq:LR} is, generally,  NP-hard (as it contains cardinality minimization as a special case, which is NP-hard \cite{FoR 13, Nat 95}),  and therefore many continuous relaxations for its numerical solution have been proposed. The predominant class of convex relaxations uses the {\em nuclear norm} (or {\em trace norm}) $\|\cdot\|_*$ as a convex approximation of the rank function. The justification for this stems from the fact that   the nuclear norm  is  the {\em convex envelope} (i.e.  the largest convex minorant) of   the rank function when restricted to a  spectral norm ball around the point in question, a fact that was first established by Fazel in her thesis \cite{Faz 02} (see also the approach by Hiriart-Urruty and Len \cite{HUL 13}).  On the other hand, the nuclear norm  is simply the $\ell_1$-norm of the vector of singular values, and the $\ell_1$-norm is known to promote sparsity \cite{CaT 05}, hence the nuclear norm promotes low rank.    Various nuclear norm-based approximations of problem \eqref{eq:LR} have been proposed, the most obvious one being  
\begin{equation} \label{eq:NNM} 
\min_{X\in \R^{n\times p}} \|X\|_*\st \cA(X)=b.
\end{equation}
Existence of solutions for this problem\footnote{Of course, we assume throughout that this problem is feasible.} is readily established as the objective function is {\em coercive}  (and the suitable continuity properties are satisfied).  Given a solution $\bar X$ of \eqref{eq:NNM}, the goal  of this paper is to establish conditions that  guarantee that $\bar X$ is, in fact, the unique solution.   This is inspired by  the study by Zhang et al. \cite{{ZYC 15}} which establishes  uniqueness results for $\ell_1$-minimization problems\footnote{Nuclear norm minimization contains $\ell_1$-minimization as a special case since $x\in \R^n$ can be identified with a diagonal matrix  $\diag(x)$ for which $\|\diag(x)\|_*=\|x\|_1$.}

We approach this task by combining tools from convex analysis and linear algebra.  The natural interplay of these areas is most obvious in the study of {\em unitarily invariant norms} \cite{HoJ 13} which comes into play here since the nuclear norm (and its dual norm, the spectral norm) are unitarily invariant.  This theory  goes back to work of von Neumann's \cite{Neu 37}, expanded on by various authors including Watson \cite{Wat 92, Wat 93}, Zietak \cite{Zie  88, Zie 93} and  de S\'a \cite{Sa 94, Sa 94.2}, and then   vastly generalized beyond norms in  Lewis'   seminal work \cite{Lew 95, Lew 96, LeS 05}.  

\subsection*{Contributions}

Our first  main contribution, \Cref{th:FlatNN}, provides  a characterization of the existence of line segments ({\em flatness}) in the boundary of the nuclear norm ball,  based on the notion of {\em simultaneous polarizability} (\Cref{def:Polar}).   In \Cref{cor:Equiv} we give a reformulation of this characterization using the singular value decomposition of a point in the nuclear norm sphere, and this directly carries over to \emph{necessary and sufficient} conditions for uniqueness (\Cref{cor:UniMin}) for solutions of the nuclear norm minimization problem \eqref{eq:NNM}.

We then extend the study by Zhang et al. \cite{{ZYC 15}} to the nuclear norm setting, starting from the following observation of Gilbert's \cite{Gil 17} for any (proper) convex function $f$ (see  \Cref{prop:CXV}):  $\bar x$ is the unique minimizer of $f$ if  $0$ is in the interior of the subdifferential of $f$ at $\bar x$. We make these conditions concrete for  problem  \eqref{eq:NNM} in \Cref{prop:Suff}. We then bridge between  these convex-analytic  conditions and  the linear-algebraic ones established earlier in  \Cref{cor:UniMin}  explicitly in  \Cref{prop:Bridge}, thus illuminating their connection. By means of a counterexample (\Cref{ex:Counter}) we show that the sufficient conditions (\Cref{ass:Suff}) are not necessary for uniqueness, which is in contrast to the (polyhedral convex) $\ell_1$-case.

Through convex analysis (\Cref{prop:FR}) we are able to transfer our findings for problem \eqref{eq:NNM} to another class of nuclear norm minimization problems (see \Cref{cor:Other} ) including nuclear  norm-regularized least-squares.

\subsection*{Roadmap} We present  in \Cref{sec:Prelim} the necessary background from linear algebra and convex analysis, including a novel result on the convex geometry of the subdifferential of the nuclear norm. \Cref{sec:Flat} is devoted to characterizing the existence of line segments in the nuclear norm sphere. We transfer these findings to nuclear norm minimization problems in \Cref{sec:NNO}. We close out with some final remarks in \Cref{sec:Final}.


\subsection*{Notation}The vector $e_i\in \R^n$ is the $i$-th standard unit vector in $\R^n$. For a vector $x\in \R^n$, $\diag(x)$ will be a diagonal matrix with $x$ on its diagonal, whose size will be clear from the context (and which may be rectangular). For $X\in \R^{n\times p}$, we will generate the vector of its diagonal entries via $\Diag(X)$. The space of $n\times n$ (real) symmetric matrices is denoted by $\bS^n$,  $\bS^n_+$ is  the positive semidefinite cone while $\bS^n_{++}$ denotes the positive definite matrices in $\bS^n$. 
 The set of $n\times n$ orthogonal matrices is denoted by $O(n)$.  For a set $C$ in a real vector space, we define  $\R_+C\coloneqq \set{tx}{t\geq 0, \; x\in C}$, the smallest cone that contains $C$. The line segment between two points $x,y$ in a real vector space  is denoted by $[x,y]$. The set of all linear maps between two Euclidean  spaces $V,W$ is denoted by $\cL(V,W)$. For $\cA\in \cL(V,W)$, we write $\ker \cA$ and $\rge \cA$ for its {\em kernel} and {\em range}, respectively. Its adjoint map is denoted by $\cA^*$.

\section{Preliminaries}\label{sec:Prelim}

\noindent
In what follows, $\bE$ will be a Euclidean space, i.e.\ a finite-dimensional real inner product space with its ambient inner product denoted by $\ip{\cdot}{\cdot}$.  The induced norm is denoted by $\|\cdot\|$, i.e.\ $\|x\|:=\sqrt{\ip{x}{x}}$ for all $x\in \bE$.  For instance, we equip $\R^{n\times p}$ with the (Frobenius) inner product 
\[
\ip{X}{Y}:=\tr(X^TY)\quad \forall X,Y\in \R^{n\times p},
\]
which  induces the {\em Frobenius norm} 
\[
\|X\|:=\sqrt{\ip{X}{X}}=\sqrt{\sum_{i=1}^n\sum_{j=1}^px_{ij}^2}\quad \forall X\in \R^{n\times p}.
\]
For $X\in \R^{n\times p}$ its {\em nuclear norm} is given by 
\[
\|X\|_*:=\tr(\sqrt{X^TX})=\tr(\sqrt{XX^T}).
\]
The definition implies  the following fact  used frequently in our study:
\begin{equation}\label{eq:SPD}
\|X\|_*=\tr(X)\quad \forall X\in \bS^n_+.
\end{equation}
The {\em dual  norm} of the  nuclear norm  is
\[
\|X\|_{op}:=\max_{\|Y\|_*\leq 1} \ip{X}{Y}=\max_{\|v\|\leq 1}\|Xv\|,
\]
which is called  the {\em operator norm} or {\em spectral norm}. In what follows, we will define
\[
\bB_{op}:=\set{X}{\|X\|_{op}\leq 1}
\]
to be the  operator norm unit ball in a matrix space  whose dimension will be  clear from the context.  The following  simple estimate for the operator norm will be useful for our study. 

\begin{lemma}\label{lem:Op} For $A\in \R^{n\times p}$. Then the Euclidean norm of every  column and row of $A$ is bounded above by $\|A\|_{op}$. In particular, we have 
$
a_{ij}\leq \|A\|_{op}$ for all  $i=1,\dots,n,\; j=1,\dots,p.$
\end{lemma} 
\begin{proof} Let $a_j$ be the  $j$-th column of $A$. Then
\[
\|a_j\|=\|Ae_j\|\leq \sup_{\|x\|=1}\|Ax\|=\|A\|_{op}.
\]
Multiplying standard unit vectors $e_j^T$ from the left, we get the analogous statement for rows. 
\end{proof}

\noindent
The following estimate for the nuclear norm of block matrices is important to our study.

\begin{lemma}\label{lem:Embed} Let $n> p$, $X\in \R^{n\times p}$ and  $Y\in \R^{n\times (n-p)}$. Then 
\[
\|X\|_*\leq \|[X\; Y]\|_*,
\]
where equality holds if and only if $Y = 0.$
\end{lemma}
\begin{proof} Observe that $\|W\|_{op}=\|[W\; 0]\|_{op}$. Hence 
\[
\cC:=\set{[W\; 0]\in \R^{n\times n}}{ W\in \R^{n\times p},\;\|W\|_{op}\leq 1}\subset \bB_{op}. 
\]
Consequently
\[
\|[X\; Y]\|_*=\max_{[W\; Z]\in \bB_{op}}\ip{[X\; Y]}{[W\;Z]}\geq \max_{[W\; 0]\in \cC}\ip{[X\; Y]}{[W\;0]}=\|X\|_*.
\]
Clearly, the inequality is strict if $Y\neq 0$ (use, e.g.,   $Z=Y$) and an equality otherwise.
\end{proof}

\noindent
We point out that the above result allows one to always embed  problem \eqref{eq:NNM} (defined by $\cA\in \cL(\R^{n\times p},\bE)$ and $b\in \bE$) in (potentially rectangular) matrix space $\R^{n\times p}$ (w.l.o.g. $n\geq p$) into the (square) matrix space $\R^{n\times n}$. To this end, identify every element $\tilde X\in \R^{n\times n}$ with the block matrix   $\tilde X=[X\;Y]$ for   $X\in \R^{n\times p}, Y\in \R^{n\times (n-p)}$, define the linear operator $\tilde \cA:\tilde X\to \cA(X)$ and the right-hand side $\tilde b:=b$. If  we now  consider the `padded' problem
\begin{equation}\label{eq:Pad}
\min_{[X\; Y]\in \R^{n\times n}} \|[X\;Y]\|_*\st \tilde \cA([X\; Y])=\tilde b
\end{equation}
it is an immediate consequence of \Cref{lem:Embed} that $\bar X$ is a solution of \eqref{eq:NNM} if and only if $[\bar X\; 0]$ is a solution of \eqref{eq:Pad}.

\subsection*{Singular value decomposition}  For the facts and concepts presented in this paragraph we refer the uninitiated reader to  Horn and Johnson \cite{HoJ 13} for details. 
Throughout  (w.l.o.g.) we assume that  $n\geq p$. For  $X\in\R^{n\times p}$,  with $\rank X=r$, there exist   orthogonal matrices $U\in O(n)$ and $V\in O(p)$ (with columns $u_1,\dots,u_n$ and $v_1,\dots,v_p$, respectively)    and unique  real numbers 
\[
\sigma_1(X)\geq \sigma_2(X)\geq \sigma_{r}(X)>0=\sigma_{r+1}=\dots=\sigma_n(X)
\]
such that
\[
X=U\diag(\sigma(X))V^T=\sum_{i=1}^r \sigma_i(X)u_iv_i^T.
\]
This is called a {\em singular value decomposition}  (SVD) of $X$. Note that the  positive singular values of $X$ are exactly the square roots of the  nonzero eigenvalues of $XX^T$ (or $X^TX$).
We say that two matrices $X,Y\in \R^{n\times p}$ have a {\em simultaneous singular value decomposition} if there exist $(\bar U,\bar V)\in O(n)\times O(p)$ such that
\[
X=\bar U\diag(\sigma(X))\bar V^T\AND Y=\bar U\diag(\sigma(Y))\bar V^T.
\]
The next result, see e.g.\ \cite[Theorem 2.1]{Lew 95},  due to von Neumann,  characterizes simultaneous singular value decompositions.
\begin{theorem}[von Neumann]\label{th:Neumann} For $X,Y\in \R^{n\times p}$  we have 
\[
\ip{X}{Y}\leq \ip{\sigma(X)}{\sigma(Y)}.
\]
Equality holds if and only if $X$ and $Y$ have simultaneous singular value decompositions.
\end{theorem}

\noindent
Through the singular value decomposition, we  generate the  map
\[
\sigma:X\in \R^{n\times p}\to \sigma(X)\in\R^p.
\]
Using this, the nuclear and operator norm of $X$, respectively, can be expressed as the $\ell_1$- and $\ell_\infty$-norm, respectively,  of the vector of singular values of $X$, i.e. 
\[
\|X\|_*=\sum_{i=1}^r \sigma_{i}(X)\AND \|X\|_{op}=\sigma_1(X).
\]
Moreover, we find that the nuclear and the operator norm are {\em orthogonally invariant}, i.e. for all $X\in \R^{n\times p}$, we have
\begin{equation}\label{eq:OI}
\|UXV\|_*=\|X\|_*\AND \|UXV\|_{op}=\|VXU\|_{op}\quad \forall (U,V)\in O(n)\times O(p).
\end{equation}

\noindent
There is an important extension of the above equation in the rectangular case.  To formulate it, we recall the {\em Stiefel manifold} \cite{Sti 36}.
\begin{definition}[Stiefel manifold]\label{def:Stiefel} 
  The {\em Stiefel manifold} $\mathcal{V}_{n,p}$ is the collection of matrices in $\R^{n \times p}$ with orthonormal columns, i.e.
  \[
  \cV_{n,p}:=\set{U\in \R^{n\times p}}{U^TU=I_p}.
  \]
\end{definition}
\noindent The nuclear norm also has invariance on one side by multiplication by elements of the Stiefel manifold.
\begin{lemma}\label{lem:OI} Let $X\in \R^{n\times p}$ and $U\in \mathcal{V}_{n,p}$. Then 
$
\|XU^T\|_*=\|X\|_*.
$
\end{lemma}
\begin{proof} 
As $U$ has orthonormal columns, we may extend it to an orthogonal matrix $[U\; W]\in  O(n)$. Then
\[
\|XU^T\|_*=\left\|[X\; 0] \cdot\left[ \begin{matrix}U^T\\W^T
\end{matrix}\right]
\right\|_*=\|[X\; 0]\|_*=\|X\|_*,
\]
where the second identity uses the orthogonal invariance from \eqref{eq:OI} and the third is due to \Cref{lem:Embed}. 
\end{proof}

\subsection*{Tools from convex analysis}  For the facts and concepts presented in this paragraph we refer the uninitiated reader to the textbooks by Rockafellar \cite{Roc 70}, Hiriart-Urruty and Lemar\'echal \cite{HUL 01}, Borwein and Lewis \cite{BoL 00} or Rockafellar and Wets \cite[Chapter 11]{RoW  98}.

A  function $f:\bE\to\rp$ is called proper if $\dom f:=\set{x}{f(x)<+\infty}\neq\emptyset$. We say that $f$ is {\em convex} if its epigraph $\epi f:=\set{(x,\alpha)\in \bE\times\R}{f(x)\leq \alpha}$ is convex, and we say that it is {\em closed} if $\epi f$ is closed. Its (Fenchel) conjugate $f^*:\bE\to\rp$ is
$
f^*(y):=\sup_{x\in \dom f} \{\ip{y}{x}-f(x)\}.
$
Its (convex) subdifferential at $\bar x\in \dom f$ is given by
 \[
 \p f(\bar x):=\set{y\in \bE}{f(\bar x)+\ip{y}{x-\bar x}\leq f(x)\;\forall x\in \dom f}.
 \]
 An important (proper, convex) extended real-valued function is the {\em indicator function}  of a (nonempty, convex) set $C\subset \bE$ which is
 \[
\delta_C:\bE\to\rp,\quad\delta_C(x)=\begin{cases}0,& x\in C,\\ +\infty,& {\rm else}. \end{cases}
\]
Its subdifferential is  $\p\delta_C(\bar x)=\set{v}{\ip{v}{x-\bar x}\leq 0\; \forall x\in C}$ for all $\bar x\in C$. Its conjugate is the support function of $C$, i.e. $\delta_C^*(y)=\sup_{x\in C}\ip{x}{y}=:\sig_C(y).$ We point out that the support functions 
of compact, convex, symmetric sets $C$ that contain $0$ (thus $0\in\inter C$) are exactly the norms on $\bE$ \cite[Theorem 15.2]{Roc 70} .

The subdifferential of the  $\ell_1$-norm $\|\cdot\|_1:\R^n\to\R,\;\|x\|_1=\sum_{i=1}|x_i|$, reads 
  \begin{equation}\label{eq:SubL1}
\p\|\cdot\|_1(x)= \left(\begin{cases} \sgn(x_i),& x_i\neq 0\\
[-1,1],& x_i=0
\end{cases}\right)_{i=1}^n=\set{y\in \bB_\infty}{\ip{x}{y}=\|x\|_1}.
\end{equation}
Obviously, the most important example to our study is the subdifferential of the nuclear norm.  
 \begin{proposition}[Subdifferential of nuclear norm] \label{prop:SD} Let $\bar X\in \R^{n\times p}$ and let 
$(\bar U,\bar V)\in O(n)\times O(p)$ such that 
\[
\bar U\diag(\sigma(\bar X))\bar V^T=\bar X.
\] 
The following hold: 
\begin{itemize}
\item[(a)] We have $Y\in \p \|\cdot\|_*(\bar X)$ if and only if  $\bar X$ and $Y$ have a simultaneous singular value decomposition and 
$\sigma(Y)\in \p\|\cdot\|_1(\sigma(\bar X))$.
\item[(b)] It holds that
\begin{eqnarray}
\p\|\cdot\|_*(\bar X)& = & \set{Y}{\ip{\bar X}{Y}=\|\bar X\|_*,\; \|Y\|_{op}\leq  1}\label{eq:SD1}\footnotemark\\ 
& = & \bar U\p \|\cdot\|_*(\diag(\sigma(\bar X))\bar V^T.\label{eq:SD3}
\end{eqnarray}
\end{itemize}
\end{proposition}
\begin{proof} (a)  See \cite[Corollary 2.5]{Lew 95}.
\smallskip

\noindent
(b) The  expressions  for the subdifferential   can be found in \cite{Zie 93}, see, in particular,  \cite[Theorem 3.1]{Zie 93} for the characterization in \eqref{eq:SD3}. 

\end{proof}

\noindent
For a convex set $C\subset \bE$,  its {\em affine hull}, denoted by $\aff C$, is the smallest affine set that contains $C$. In particular, $\aff C$ is a subspace if and only if it contains $0$. The {\em subspace parallel to $C$} is defined to be the unique subspace parallel to $\aff C$ and  given by $\para C\coloneqq\aff C -\bar x$ for any $\bar x\in  C$.  Clearly, this entails that $\ri C=\inter C$ if and only if the latter is nonempty, i.e. when $\para C=\bE$.

The  {\em relative interior} $\ri C$ of $C$ is  its  interior in the relative topology with respect to its affine hull. The following characterization of relative interior points is useful to our study, see, e.g., \cite[Exerc.~13, Ch.~1]{BoL 00}: 
\begin{equation}\label{eq:RelInt}
x\in \ri C \IFF  \R_{+}(C-x)=\para C.
\end{equation}
For more details on the relative interior we refer the reader to Rockafellar \cite[Chapter 6]{Roc 70}.

We will now exploit the representation in \eqref{eq:SD3} to derive yet another representation of the subdifferential of the nuclear norm as well as its relative interior and parallel subspace. This  is useful to  our study but also of independent interest. We need the following lemma.
\begin{lemma}\label{lem:SDAux}
For  $r\leq p(\leq n)$ set
\[
\cT:=\set{B\in \R^{n\times p}}{\Diag(B)\in \{1\}^r\times \R^{n-r},\; \|B\|_{op}\leq 1}.
\]
Then
\[
\cT=\set{\left(\begin{smallmatrix}I_r & 0 \\ 0 &  R\end{smallmatrix}\right)}{R\in \R^{(n-r)\times (p-r)},\;\|R\|_{op}\leq 1}.
\]
\end{lemma}

\begin{proof} Let $B\in\cT$. Then, by \Cref{lem:Op} and the fact that $ b_{ii}=1$ for all $i=1,\dots, r$, we find that $B=\left(\begin{smallmatrix}I_r & 0 \\ 0 &  R\end{smallmatrix}\right)$ for some $R\in \R^{(n-r)\times (p-r)}$. Now observe that
 \[
 \|R\|_{op}\leq 1 \IFF \left\|\left(\begin{smallmatrix}I_r & 0 \\ 0 &  R\end{smallmatrix}\right)\right\|_{op}\leq 1.
 \]
 This shows the desired equality.
\end{proof}

\begin{proposition}[Convex geometry of $\p\|\cdot\|_*(\bar X)$]\label{prop:SD1} Let $\bar X\in \R^{n\times p}$ with $r:=\rank \bar X$ and let 
$(\bar U,\bar V)\in O(n)\times O(p)$ such that 
\[
\bar U\diag(\sigma(\bar X))\bar V^T=\bar X.
\] 
Then the following hold:
\begin{itemize}
\item[(a)] $
\p \|\cdot\|_*(\bar X)=\bar U\set{\left(\begin{smallmatrix}I_r & 0 \\ 0 &  R\end{smallmatrix}\right)}{R \in \R^{(n-r)\times (p-r)},\;\|R\|_{op}\leq 1}\bar V^T;
$
\item[(b)] $\ri(\p\|\cdot\|_*(\bar X))=\bar U\set{\left(\begin{smallmatrix}I_r & 0 \\ 0 &  R\end{smallmatrix}\right)}{R \in \R^{(n-r)\times (p-r)},\;\|R\|_{op}<1}\bar V^T$;
\item[(c)] $\para (\p\|\cdot\|_*(\bar X))=\bar U\set{\left(\begin{smallmatrix}0 & 0 \\ 0 &  R\end{smallmatrix}\right)}{R\in \R^{(n-r)\times (p-r)}}\bar V^T$.

\end{itemize}
\end{proposition}
\begin{proof} (a) From the characterization \eqref{eq:SD3}  we find that 
$
\bar U\p \|\cdot\|_*(\diag(\sigma(\bar X))\bar V^T.
$
In turn, by \eqref{eq:SD1},  we find that 
\[
\p \|\cdot\|_*(\diag(\sigma(\bar X))=\set{B\in \R^{n\times p}}{\ip{B}{\diag(\sigma(\bar X)}=\|\diag(\sigma(\bar X))\|_*,\;\|B\|_{op}\leq 1}.
\]
Now, observe that, by taking adjoints, 
$
\ip{B}{\diag(\sigma(\bar X)}=\ip{\Diag(B)}{\sigma(\bar x)},
$
and also $\|\diag(\sigma(\bar X))\|_*=\|\sigma(\bar X)\|_1$. Hence 
\[
\p \|\cdot\|_*(\diag(\sigma(\bar X))  =  \set{B\in \R^{n\times p}}{\Diag(B)\in \{1\}^r\times \R^{n-r},\; \|B\|_{op}\leq 1},
\]
and thus \Cref{lem:SDAux} gives the desired result.
\smallskip

\noindent
(b) Define $F:\R^{(n-r)\times(p-r)}\to \R^{n\times n}$ by 
$
F(R)=\bar U\left(\begin{smallmatrix}I_r & 0 \\ 0 &  R\end{smallmatrix}\right)\bar V^T. 
$
Then,  in view of (a), we find that $\p\|\cdot\|_*(\bar X)=F(\bB_{op})$. Therefore the desired formula follows from \cite[Theorem 6.6]{Roc 70}.
\smallskip

\noindent
(c) Follows immediately from (a) or  (b).
\end{proof}

\noindent
In the setting of \Cref{prop:SD1}, it follows immediately from part (c) that 
\begin{equation}\label{eq:Parallel}
\para (\p\|\cdot\|_*(\bar X))=\lin\set{u_iv_j^T}{i=r+1,\dots,n,\; j=r+1,\dots,p},
\end{equation}
where $u_i\;(i=1,\dots,n)$ and $v_j\;(j=1,\dots,p)$ are the columns of $\bar U$ and $\bar V$, respectively.

%
%
%

\section{Flatness of the nuclear norm and simultaneous polarizability} \label{sec:Flat}

\noindent
As before, we assume (w.l.o.g.) that $n\geq p$.
In this section we present our main results on the geometry of the nuclear norm sphere, specifically a characterization of the {\em flats}\footnote{Flats, in the context of Riemannian geometry, are (uncurved) Euclidean submanifolds.}.  We then leverage this to characterize the uniqueness of certain nuclear norm optimization problems.
The next definition is central to this analysis.
\begin{definition}[Polarizability]\label{def:Polar} 
Let  $X,\hat X\in \R^{n\times p}$. 
\begin{itemize}
  \item[(a)]   We say that $U\in \mathcal{V}_{n,p}$ {\em polarizes}\footnote{Sometimes this is also called the `angular' part of the polar decomposition.} $X$ if $XU^T\in \bS^n_+$.
  \item[(b)] We say that  $X$ and $\hat X$ are   {\em simultaneously polarizable}  if there exists a matrix $U\in \mathcal{V}_{n,p}$ that polarizes both $ X$ and $\hat X$.
\end{itemize}

\end{definition}

\noindent
A polarization in the sense of \Cref{def:Polar} (a) always exists as the following result shows, which is based on {\em polar decomposition}. 
Note that conventionally, for the case of the rectangular polar decomposition, the polarizing matrix $U$  usually appears on the small side of the matrix, see Horn and Johnson \cite[Theorem 7.3.1]{HoJ 13}.  In  \Cref{def:Polar}, we have placed it on the large side, but we observe that by padding, it is possible to conclude the existence of the large polarization as well.
\begin{proposition}[Existence of polarization]\label{prop:Polar} Let $X\in \R^{n\times p}$. Then there exists  $U\in \mathcal{V}_{n,p}$ that polarizes $X$.
\end{proposition}
\begin{proof} Consider the augmented matrix $[X\;  0]\in \R^{n\times n}$. By polar decomposition, see, e.g.,  \cite[Theorem 7.3.1]{HoJ 13}, there exists $Q\in O(n)$ and $S\in \bS^n_{+}$ such that  $[X\; 0] =S Q$. Now, partition $Q=[U\; W]$ according to $[X\;0]$. Then 
\[
XU^T=[X\;0] \cdot \left[ \begin{matrix}U^T\\W^T
\end{matrix}\right]
=S\succeq 0,
\]
and, by construction, $U$ has orthonormal columns.
\end{proof}

\noindent
Polarizability can be expressed in terms of the subdifferential of the nuclear norm.

\begin{lemma}\label{lem:Polar} Let $X\in \R^{n\times p}$ and let $U\in \mathcal{V}_{n,p}$.
 The following are equivalent:
\begin{itemize}
\item[(i)] $U\in \p \|\cdot\|_*(X)$;
\item[(ii)] $\ip{U}{X}=\|X\|_*$;
\item[(iii)]  $U$ polarizes $X$, i.e. $XU^T\in\bS^n_+$.
\end{itemize}
\end{lemma}
\begin{proof} (i)$\Rightarrow$(ii): By   the  subdifferential representation   of $\|\cdot\|_*$ in \eqref{eq:SD1}.
\smallskip

\noindent
(ii)$\Rightarrow$(iii):  Observe that $U^TU=I_p$. In particular,  $\sigma(U)=[1,\dots, 1]^T\in \R^p$. By assumption,  we hence have 
$
\ip{U}{X}=\|X\|_*=\ip{\sigma(U)}{\sigma(X)}.
$
By (von Neumann's)  \Cref{th:Neumann}, we thus find $\bar U\in O(n), \bar V\in O(p)$ such that $X=\bar U\diag(\sigma(X))\bar V^T$ and  $U=\bar U\left(\begin{smallmatrix}I_p & 0 \\ 0 &  0\end{smallmatrix}\right) \bar V^T$. Consequently 
\[
XU^T=\bar U\left(\begin{smallmatrix}\diag(\sigma(\bar X)) & 0 \\ 0 &  0\end{smallmatrix}\right)\bar V^T\bar V\left(\begin{smallmatrix}I_p & 0 \\ 0 &  0\end{smallmatrix}\right)\bar U^T=\bar U\left(\begin{smallmatrix}\diag(\sigma(\bar X)) & 0 \\ 0 &  0\end{smallmatrix}\right)\bar U^T\succeq 0.
\]
(iii)$\Rightarrow$(i): Extend $U$ to a an orthonormal matrix $[U\; W]\in  O(n)$. Then
\[
\tr(XU^T)=\|XU^T\|_*=\|X\|_*,
\]
where the first identity employs the assumption that $XU^T\in \bS^n_+$ combined with  \eqref{eq:SPD} and the second one is due to \Cref{lem:OI}.  Since $\|U\|_{op}=1$ (as $U^TU=I_p$), the desired statement follows from \Cref{prop:SD}(b).
\end{proof}

\noindent
We now present our first main result which characterizes the existence of (proper) line segments in the nuclear norm sphere.

\begin{theorem}[Flats in the nuclear norm sphere]\label{th:FlatNN}  Let $\bar X,\hat X\in \R^{n\times p}$ and define
\[
X(t):=\bar X+t(\hat X-\bar X)\quad \forall t\in[0,1]. 
\]
Then the following are equivalent:
\begin{itemize}
\item[(i)]  $\|X(t)\|_*=\|X\|_*$ for all $t\in [0,1]$.
\item[(ii)] $\hat X$ and $\bar X$ are simultaneously polarizable and $\|\bar X\|_*=\|\hat X\|_*$.

\end{itemize}

\end{theorem}

\begin{proof} Note that there is nothing to prove if $\bar X=\hat X$. So we assume the contrary from now on.
\smallskip

\noindent
(i)$\Rightarrow$(ii): By assumption,  the convex function $f:\R\to \R,$
$
f(t)=\|X(t))\|_*, 
$ 
is constant on $[0,1]$.  Hence, by the (subdifferential) chain rule \cite[Theorem 23.8]{Roc 70}, we have 
\[
\{0\}=\{f'(t)\}=\set{\ip{\hat X-\bar X}{Y}}{Y\in \p\|\cdot\|_*(X(t))}\quad\forall t\in (0,1),
\]
i.e.
\begin{equation}\label{eq:Chain}
\ip{\hat X-\bar X}{Y}=0\quad\forall Y\in \p\|\cdot\|_*(X(t)),\; t\in (0,1).
\end{equation}
Now, for any $t\in (0,1)$ and any $Y\in \p\|\cdot\|_*(X(t))$, we have 
\begin{equation}\label{eq:Xt}
\ip{t\hat X+(1-t)\bar X}{Y}=\ip{X(t)}{Y}=\|X(t)\|_*=\|\bar X\|_*.
\end{equation}
Multiplying \eqref{eq:Chain} by $-t$ and adding  to \eqref{eq:Xt} then yields
\[
\ip{\bar X}{Y}=\|\bar X\|_* \quad \forall Y\in \p\|\cdot\|_*(X(t)),\; t\in (0,1),
\]
hence 
\begin{equation}\label{eq:XbarCont}
\p \|\cdot\|_*(X(t))\subset \p \|\cdot\|_*(\bar X)\quad \forall t\in(0,1).
\end{equation}
Similarly,  multiplying  \eqref{eq:Chain} by $(1-t)$ and adding  to \eqref{eq:Xt} ultimately yields
\begin{equation}\label{eq:XhatCont}
\p \|\cdot\|_*(X(t))\subset \p \|\cdot\|_*(\hat X)\quad \forall t\in(0,1).
\end{equation}
Combining  \eqref{eq:XbarCont} and \eqref{eq:XhatCont} we thus find 
\begin{equation}\label{eq:Inter}
\p \|\cdot\|_*(X(t))\subset \p \|\cdot\|_*(\hat X)\cap \p \|\cdot\|_*(\bar X)\quad \forall t\in(0,1).
\end{equation}
Now, for $t\in (0,1)$, set $X_t:=X(t)$. Choose $U_t\in \R^{n\times p}$ that polarizes  $X_t$ by means of \Cref{prop:Polar}. Then, by \Cref{lem:Polar} , we have  $U_t\in \p \|\cdot\|_*(X_t)$, and consequently, by \eqref{eq:Inter}, we find $U_t\in  \p \|\cdot\|_*(\hat X)\cap \p \|\cdot\|_*(\bar X)$. Therefore, we find
\[
\ip{\bar X}{U_t}=\|\bar X\|_*\AND \ip{\hat X}{U_t}=\|\hat X\|_*.
\]
By \Cref{lem:Polar}  we thus infer that $U_t$ polarizes both $\bar X$ and $\hat X$.
\smallskip

\noindent
(ii) $\Rightarrow$ (i): Let $U\in \R^{n\times p}$  polarize $\bar X$ and $\hat X$. Consequently,  $U$ polarizes $X(t)$, and hence, by \Cref{lem:Polar}, $U\in \p\|\cdot\|_*(X(t))$  for all $t\in [0,1]$. Therefore
\[
\|X(t)\|_*=\tr(X(t)U^T)=t\cdot \tr(\bar XU^T)+(1-t)\cdot\tr(\hat XU^T)=\|\bar X\|_*.
\] 
Here,   the last identity uses that $ \tr(\bar XU^T)=\|\bar X\|_*=\|\hat X\|_*= \tr(\hat XU^T)$ as $U$ polarizes both $\bar X$ and $\hat X$ which  have the same nuclear norm (by assumption).
\end{proof}  

\vspace{0.5cm}

\noindent
An immediate consequence is the following corollary.

\begin{corollary}\label{cor:FlatNN1} For  $\bar X,\hat X\in \R^{n\times p}$ and $\cA\in \cL(\R^{n\times p}, \bE)$ the following are equivalent:
\begin{itemize}
\item[(i)] $\|\cdot\|_*$ and $\cA$ are constant on the line segment $[\bar X,\hat X]$.
\item[(ii)] $\hat X-\bar X\in \ker \cA$, $\|\bar X\|_*=\|\hat X\|_*$ and $\hat X$ and $\bar X$ are simultaneously polarizable.

\end{itemize}
\end{corollary}

\noindent
The previous result, while geometrically elegant, is potentially difficult to evaluate.  By working with  the singular value decomposition  of a the base point  $\bar X$, one can further specify exactly the set of directions which should not be contained in  the kernel of the ambient linear operator $\cA$.
To state this result, we use the following notation for some $r\in \{1,\dots,n\}$:
\[
\cS_{++}^r \coloneqq \set{\left(\begin{smallmatrix}A& 0\\ 0 & 0\end{smallmatrix}\right)\in \bS^n_+}{A\in \bS^r_{++}}\subset \bS^n_+.
\]

\begin{corollary}\label{cor:Equiv}
 Let $\bar X\in \R^{n \times p}$ and let $r:=\rank \bar X$.  Let there be posed a singular value decomposition $\bar X=\bar U\diag(\sigma(\bar X))\bar V^T$ and let  $\cA\in \cL(\R^{n\times p}, \bE)$. Set 
\[
W(\bar X) \coloneqq
\left\{
  \bar UM\left(\begin{smallmatrix}I_r & 0\\ 0 & R\end{smallmatrix}\right)\bar V^T
  \;
\middle\vert
\;
\begin{aligned}
&M\in \bS^n_+-\cS^r_{++}, \; \tr(M)=0,\\
&R \in \mathcal{V}_{n-r,p-r}, \;M\left(\begin{smallmatrix}I_r & 0\\ 0 & RR^T\end{smallmatrix}\right)=M
\end{aligned}
\right\}.
\]
See \Cref{sec:W} for a discussion of $W$.
The following are equivalent:
\begin{itemize}
\item[(i)] $\cX:=\set{X\in \R^{n\times n}}{\cA(X)=\cA(\bar X),\; \|X\|_*=\|\bar X\|_*}$ does not contain a proper\footnote{A line segment that is not just $\{\bar X\}$.} line segment including $\bar X$.
\item[(ii)] $\ker \cA\cap W(\bar X)=\{0\}$.

\end{itemize}

\end{corollary}
\begin{proof} 
  (ii)$\Rightarrow$(i): Assume (i)  does not hold, i.e.\ there  is $\hat X\neq \bar X$ such that $[\bar X,\hat X]\subset\cX$. By \Cref{cor:FlatNN1}, we find that $\bar X$ and $\hat X$ are simultaneously polarizable, i.e. there exists $U\in \mathcal{V}_{n,p}$ such that $\bar XU^T\in \bS^n_+$ and $\hat XU^T\in \bS_+^n$. In particular, by \Cref{lem:Polar}, $U\in \p \|\cdot\|_*(\bar X)$, hence, by \Cref{prop:SD1}, $U=\bar U\left(\begin{smallmatrix}I_r & 0 \\ 0 &  R\end{smallmatrix}\right)\bar V^T$ for some $R\in \mathcal{V}_{n-r,p-r}$. The latter comes from the fact that $U\in \cV_{n,p}$.
 Moreover, since $\hat XU^T\in \bS_+^n$, we find that
\begin{equation}\label{eq:SPDhat}
\bar U^T\hat X\bar V\left(\begin{smallmatrix}I_r & 0 \\ 0 &  R^T\end{smallmatrix}\right)=   \bar U^T(\hat XU^T) \bar U\in \bS^n_+.
\end{equation}
On the other hand, we also have 
\[
\bar U^T\bar X\bar V\left(\begin{smallmatrix}I_r & 0 \\ 0 &  R^T\end{smallmatrix}\right)=\diag(\sigma(\bar X))\in \cS^r_{++}.
\]
Combining this with \eqref{eq:SPDhat}, we find that 
\[
M:=\bar U^T(\hat X-\bar X)\bar V\left(\begin{smallmatrix}I_r & 0 \\ 0 &  R^T\end{smallmatrix}\right)\in \bS^n_+-\cS^r_{++},
\]
and, trivially, $M\left(\begin{smallmatrix}I_r & 0\\ 0 & RR^T\end{smallmatrix}\right)=M$.
Moreover
\begin{eqnarray*}
\tr(M)& = & \tr\left(\bar U^T\hat X\bar V\left(\begin{smallmatrix}I_r & 0 \\ 0 &  R^T\end{smallmatrix}\right)\right)-\tr\left(\bar U^T\bar X\bar V\left(\begin{smallmatrix}I_r & 0 \\ 0 &  R^T\end{smallmatrix}\right)\right)\\
& = & \|\bar U^T\hat X\bar V\left(\begin{smallmatrix}I_r & 0 \\ 0 &  R^T\end{smallmatrix}\right)\|_*-\|\bar U^T\bar X\bar V\left(\begin{smallmatrix}I_r & 0 \\ 0 &  R^T\end{smallmatrix}\right)\|_* \\
& = & \|\hat X\|_*-\|\bar X\|_*\\
& = & 0,
\end{eqnarray*}
where the second identity uses the positive semidefiniteness of the matrices in question (combined with \eqref{eq:SPD}), and the third one uses orthogonal invariance and \Cref{lem:OI}. Since also $\cA(\hat X)=\cA(\bar X)$, we consequently  have 
\[
0\neq \hat X-\bar X=\bar UM\left(\begin{smallmatrix}I_r & 0\\ 0 & R\end{smallmatrix}\right)\bar V^T\in W(\bar X)\cap \ker\cA.
\]
(i)$\Rightarrow$(ii): Assume (ii)  does not hold, i.e. there exists $M\in  (\bS^n_+-\cS^r_{++})\setminus\{0\}$ with $\tr(M)=0$ and $R\in \mathcal{V}_{n-r,p-r}$ such that  $Y\coloneqq\bar UM\left(\begin{smallmatrix}I_r & 0\\ 0 & R\end{smallmatrix}\right)\bar V^T\in \ker \cA$. Define $U:=\bar U\left(\begin{smallmatrix}I_r & 0 \\ 0 &  R\end{smallmatrix}\right)\bar V^T$. Since $\|R\|_{op}=1$, in view of \Cref{prop:SD1}, we have $U\in \p\|\cdot\|_*(\bar X)$. Now, for $\varepsilon>0$ set $X(\varepsilon)\coloneqq \bar X+\varepsilon Y$ . Then 
\begin{eqnarray*}
\bar U^T(X(\varepsilon) U^T)\bar U& =  & \bar U^T\bar XU^T\bar U+\varepsilon \bar U^TYU^T\bar U\\
& = & \diag(\sigma(\bar X))+\varepsilon M\left(\begin{smallmatrix}I_r & 0 \\ 0 &  RR^T\end{smallmatrix}\right)\\
& = &  \diag(\sigma(\bar X))+\varepsilon M.
\end{eqnarray*}
Recall  that $M\in \bS^n_+-\cS^r_{++}$, and $\diag(\sigma(\bar X))\in \cS^r_{++}$, and hence we can find $\hat \varepsilon >0$, sufficiently small, such that $\diag(\sigma(\bar X))+\hat \varepsilon M\in \bS^n_+$. Consequently, for $\hat X \coloneqq X(\hat \varepsilon)$, we have 
$\hat XU^T\in \bS_+^n$, i.e. $U$ polarizes  $\hat X$ (and $\bar X$). In addition, we find that
\begin{eqnarray*}
\|\hat X\|_*=\tr(\bar U^T(\hat XU^T)\bar U)=\tr(\diag(\sigma(\bar X)))+\hat\varepsilon\cdot\tr(M)=\|\bar X\|_*,
\end{eqnarray*}
as $\tr(M)=0$. Since we have $\cA(\hat X)=\cA(\bar X)$ as well,  \Cref{cor:FlatNN1} now gives the desired conclusion.
\end{proof} 

%

\subsection{The set  $W(\bar X)$ in \Cref{cor:Equiv}}\label{sec:W}

Some comments on the contents of \Cref{cor:Equiv} are in order.
  We note that that the set $W(\bar X)$ is a cone, owing to  the set $\bS^n_+-\cS^r_{++}$ being a cone.  While the cone is not reflection symmetric, the condition (ii) is equivalently formulated with the symmetrization of $W(\bar X)$ under the reflection $x \mapsto -x.$  The symmetrization of $W(\bar X)$ has the interpretation as the subset of the tangent space at $\bar X$ (in $\R^{n \times p}$) in which the nuclear norm changes linearly (i.e.\ is non-strictly convex).  The cone $W(\bar X)$ is \emph{not} generally convex, save for the case that $\bar{X}$ is full rank; in that case, the set $W(\bar X)$ simplifies to
\begin{equation*}\label{eq:WXfull}
W(\bar X) \coloneqq
\left\{
  \bar U\;M [I_p \; 0]^T\;\bar V^T
  \;
\middle\vert
\;
\begin{aligned}
M\in \bS^n, \; \tr(M)=0
\end{aligned}
\right\}.
\end{equation*}
Moreover, we point out that in the square case ($n=p$), the set $W(\bar X)$ simplifies to 
\[
W(\bar X)= \set{\bar UM\left(\begin{smallmatrix}I_r & 0\\ 0 & R\end{smallmatrix}\right)\bar V^T}{ M\in \bS^n_+-\cS^r_{++}, \; \tr(M)=0,\; R\in O(n-r)}.
\]
\noindent While $W(\bar X)$ is relatively pathological, we note that its span has a simple expression:
\begin{proposition}[$\lin W(\bar X)$]
  Let $\bar X\in \R^{n \times p}$ and let $r:=\rank \bar X$.  Let there be posed a singular value decomposition $\bar X=\bar U\diag(\sigma(\bar X))\bar V^T$ and let  $\cA\in \cL(\R^{n\times p}, \bE)$. Let $W(\bar X)$ be as in \Cref{cor:Equiv}.  Then 
\[
  \lin
  W(\bar X)
  =
\left\{
  \bar U\left(\begin{smallmatrix}A & B\\ C & D\end{smallmatrix}\right)\bar V^T
  \;
\middle\vert
\;
A \in \bS^r, B\in \R^{r\times (p-r)}, C\in \R^{(n-r)\times r}, D\in \R^{(n-r)\times (p-r)}
\right\}.
\]

\end{proposition}
\begin{proof}
  Let $W_4$ be the right-hand side of the displayed equation.  
The containment of $\lin W(\bar X) \subset W_4$ is immediate from the containment $W(\bar X) \subset W_4$ and the fact that the latter is a subspace.
For the reverse, we argue by construction of a flag $W_1 \subset W_2 \subset W_3 \subset W_4$ each of which we show is in $\lin W(\bar X)$.  Set 
\begin{equation}\label{eq:flag}
  \begin{aligned}
  W_1 &\coloneqq 
\left\{
  \bar U\left(\begin{smallmatrix}0 & 0 \\ 0 & D \end{smallmatrix}\right)\bar V^T
  \;
\middle\vert
\;
D \in  \R^{(n-r)\times (p-r)}
\right\}, \\
  W_2 &\coloneqq 
\left\{
  \bar U\left(\begin{smallmatrix}A & 0 \\ 0 & D \end{smallmatrix}\right)\bar V^T
  \;
\middle\vert
\;
A \in \bS^{r},
D \in  \R^{(n-r)\times (p-r)}\right\}, \\
W_3 &\coloneqq 
\left\{
  \bar U\left(\begin{smallmatrix}A & 0\\ C & D \end{smallmatrix}\right)\bar V^T
  \;
\middle\vert
\;
A \in \bS^{r},
C\in \R^{(n-r)\times r}, D\in \R^{(n-r)\times (p-r)}
\right\}.
\end{aligned}
\end{equation}

\noindent \underline{$W_1 \subset \lin W(\bar X)$}: For any element of $W_1$ for some $D \in  \R^{(n-r)\times (p-r)}$, let $R \in \mathcal{V}_{n-r,p-r}$ be a polarizing matrix such that $P=DR^T \in \bS^{n-r}_+$.  
Now set
\[
  M := 
  \left(\begin{smallmatrix}aI_r & 0 \\ 0 & P\end{smallmatrix}\right)
  \quad
  \text{where}
  \quad
  a := -\frac{\tr(P)}{r}.
\]
Then
\begin{equation}\label{eq:w0}
  \tfrac12M\left(\begin{smallmatrix}I_r & 0\\ 0 & R\end{smallmatrix}\right)
 -\tfrac12M\left(\begin{smallmatrix}I_r & 0\\ 0 & -R\end{smallmatrix}\right)
 =M
 \left(\begin{smallmatrix}0 & 0\\ 0 & R \end{smallmatrix}\right)
 =
 \left(\begin{smallmatrix}0 & 0\\ 0 & D \end{smallmatrix}\right).
\end{equation}
Since $PRR^T=DR^TRR^T=DR^T=P$ and $\tr(M) = 0,$ we conclude that $W_1 \subset \lin W(\bar X)$. \\

\noindent \underline{$W_2 \subset \lin W(\bar X)$}:  It suffices to show that $W_2/W_1 \subset \lin W(\bar X)/W_1$.
For arbitrary $A \in \bS^{r}$, let $P_1,P_2 \in \bS^{r}_{++}$ be such that $A=P_1-P_2.$
Now  set
\[
  M_i := 
  \left(\begin{smallmatrix}-P_{i} & 0 \\ 0 & a_iI_{n-r}\end{smallmatrix}\right)
  \quad
  \text{where}
  \quad
  a_i = \tfrac{\tr(P_i)}{n-r}\quad i=1,2.
\]
 Then $\tr(M_i)=0\;(i=1,2)$, and taking $R=[I_{p-r}\; 0]^T$, we have that 
\[
  \left(\begin{smallmatrix}A & 0 \\ 0 & * \end{smallmatrix}\right)
  =
  -M_1
  \left(\begin{smallmatrix}I_r & 0\\ 0 & R\end{smallmatrix}\right)
  +
  M_2
  \left(\begin{smallmatrix}I_r & 0\\ 0 & R\end{smallmatrix}\right)
  \in\lin W(\bar X), 
\]
where the $(*)$ represents a matrix of no importance. Therefore, $W_2/W_1 \subset \lin W(\bar X)/W_1$, which suffices to show the desired inclusion.\\ 

\noindent \underline{$W_3 \subset \lin W(\bar X)$}:  It suffices to show that $W_3/W_2 \subset \lin W(\bar X)/W_2$.
Take an arbitary element $W_3/W_2$ represented by some matrix $C \in \R^{(n-r)\times r}.$
For some $c$ sufficiently large
\[
  A = 
  \left(\begin{smallmatrix} -\tfrac{n-r}{r}cI_r  & C^T \\ C & c I_{n-r}\end{smallmatrix}\right)
  \in 
  \{
    M \in 
    \bS^n_+-\cS^r_{++} ~\vert~ \tr(M)=0
  \}.
\]
Taking $R=[I_{p-r}\; 0]^T$, we find 
\[
  \tfrac12A\left(\begin{smallmatrix}I_r & 0\\ 0 & R\end{smallmatrix}\right)
 +\tfrac12A\left(\begin{smallmatrix}I_r & 0\\ 0 & -R\end{smallmatrix}\right)
 =A
 \left(\begin{smallmatrix}I_r & 0\\ 0 & 0 \end{smallmatrix}\right)
 =
 \left(\begin{smallmatrix}A & 0\\ C & 0 \end{smallmatrix}\right)
  \in W(\bar X),
\]
and hence $W_3/W_2 \in \lin W(\bar X)/W_2$.\\

\noindent \underline{$W_4 \subset \lin W(\bar X)$}:  It suffices to show that $W_4/W_3 \subset \lin W(\bar X)/W_3$.
Take an arbitary element $W_4/W_3$ represented by some matrix $B \in \R^{r \times (p-r)}.$
Then for all $c$ sufficiently large,
\[
  A = 
  \left(\begin{smallmatrix} -\tfrac{p-r}{r}cI_r  & B & 0 \\ B^T & c I_{p-r}& 0 \\ 0 & 0 & 0\end{smallmatrix}\right)
  \in 
  \{
    M \in 
    \bS^n_+-\cS^r_{++} ~\vert~ \tr(M)=0
  \}.
\]
Taking $R=[I_{p-r}\; 0]^T$, we find 
\[
  A\left(\begin{smallmatrix} I_r & 0\\ 0 & R\end{smallmatrix}\right)
  =
  \left(\begin{smallmatrix} -\tfrac{p-r}{r}cI_r  & B  \\ B^T & c I_{p-r} \\ 0 & 0 \end{smallmatrix}\right)  
  \in W(\bar X),
\]
and hence we conclude that $W_4/W_3 \in \lin W(\bar X)/W_3$.

This concludes the proof.
\end{proof}

%
\section{Unique solutions in nuclear norm minimization}\label{sec:NNO}

\noindent
Throughout this section, let $(\bE, \ip{\cdot}{\cdot})$ be a (finite-dimensional) Euclidean space and let
$\cA\in \cL(\R^{n\times p},\bE)$. Our study above immediately yields uniqueness results for the nuclear norm minimization 
\begin{equation}\label{eq:BP}
\min_{X\in \R^{n\times p}} \|X\|_*\st \cA(X)=b.
\end{equation} 

\begin{corollary} \label{cor:UniMin} Let $\bar X\in\R^{n\times p}$ be a solution of \eqref{eq:BP} with $\rank \bar X=r$, and let $W(\bar X)$ be defined as in \Cref{cor:Equiv}.  Then the following are equivalent:
\begin{itemize}
\item[(i)] $\bar X$ is the unique solution of \eqref{eq:BP}.
\item[(ii)] $\ker \cA\cap W(\bar X)=\{0\}$. 
\end{itemize}
\end{corollary}
\begin{proof} Observe that the solution set $\cX$ of  \eqref{eq:BP} is convex and can be written as   $\cX=\set{X\in \R^{n\times n}}{\cA(X)=\cA(\bar X),\; \|X\|_*=\|\bar X\|_*}$. Now, by convexity, $\cX$ does not contain a proper line segment including $\bar X$ if and only if $\bar X$ is the unique solution of \eqref{eq:BP}. Thus \Cref{cor:Equiv} gives the desired statement.
\end{proof}

%

\subsection{Sufficient conditions through convex analysis}

\noindent
The following result is a generic convex analysis result which, given a solution,  provides a sufficient condition for uniqueness of solutions to a(ny) convex optimization problem. It was established in \cite{Gil 17} that in the polyhedral convex case it is also necessary which was then exploited to establish uniqueness of solutions for $\ell_1$-minimization problems.

 \begin{proposition}\label{prop:CXV} Let $f:\bE\to \rp$ be proper, convex and assume that $0\in\inter \p f(\bar x)$. Then 
 $\argmin f=\{\bar x\}$.
 \end{proposition} 
\begin{proof} Let $x\in \bE$. By assumption, there exists $\varepsilon>0$ such that $\varepsilon(x-\bar x)\in \p f(\bar x)$. Consequently
\[
f(x)\geq f(\bar x)+\ip{\varepsilon(x-\bar x)}{x-\bar x}=f(\bar x)+\varepsilon \|x-\bar x\|^2>f(\bar x).
\]
\end{proof}

%
%

\noindent
We will, of course, apply this to the objective function $f=\|\cdot\|_*+\delta_{\{0\}}(\cA(\cdot)-b)$ of \eqref{eq:BP}. It turns out that the following conditions at some (feasible) point $\bar X$ are equivalent to having $0\in \inter (\p f(\bar X))$.

\begin{assumption}\label{ass:Suff} For $\bar X\in \R^{n\times p}$ such that $\cA(\bar X)=b$ it holds that:
\begin{itemize}
\item[(i)]  $\ri (\p \|\cdot\|_*(\bar X))\cap \rge \cA^*\neq\emptyset$;
\item[(ii)] $\para\left( \p \|\cdot\|_*(\bar X)\right)+\rge\cA^*=\R^{n\times p}$.
\end{itemize}
\end{assumption}


\noindent
The reader can make these conditions even more tangible by inserting the respective expressions for the relative interior and parallel subspace of the $\p \|\cdot\|_*(\bar X)$ provided in \Cref{prop:SD1} (and \eqref{eq:Parallel}). 

We now provide the advertized characterization. 

\begin{proposition}\label{prop:Suff}   Let $\cA\in \cL(\R^{n\times p},\bE)$, $b\in \bE$ and define the (closed) proper, convex function $f:\R^{n\times p}\to \rp$ by $
f(X)=\|X\|_*+\delta_{\{0\}}(\cA(X)-b).
$
For $\bar X\in\R^{n\times p}$ such that $\cA(\bar X)=b$, the following are equivalent:
\begin{itemize}
\item[(I)] $0\in \inter\p f(\bar X).$
\item[(II)] \Cref{ass:Suff} holds at $\bar X$.
\end{itemize}
\end{proposition}
\begin{proof} Observe that  $\p(\delta_{\{0\}}((\cdot)-b)\circ \cA)(\bar X)=\cA^*\p\delta_{\{0\}}(0)=\cA^*\bE=\rge \cA^*$, by the chain rule \cite[Theorem 23.9]{Roc 70}, and consequently  
$
\p f(\bar X)=\p\|\cdot\|_*(\bar X)+\rge \cA^*,
$
by the sum rule \cite[Theorem 23.8]{Roc 70}. Hence (I) reads
\[
0\in \inter(\p \|\cdot\|_*(\bar X)+\rge \cA^*)=\ri (\p \|\cdot\|_*(\bar X))+\rge\cA^*,
\]
where the identity uses the sum  rule for the relative interior  \cite[Corollary 6.6.2]{Roc 70} and the fact that a subspace is relatively open.  This already shows that (I) implies $\ri (\p \|\cdot\|_*(\bar X))\cap \rge \cA^*\neq\emptyset$.
On the other hand, it also yields that, for any $y\in \p \|\cdot\|_*(\bar X)$, we have 
\begin{eqnarray*}
\R^{n\times p} & = & \aff (\p \|\cdot\|_*(\bar X)+\rge \cA^*)\\
& = & \aff (\p \|\cdot\|_*(\bar X)+\rge\cA^*\\
&=  & \aff(\p \|\cdot\|_*(\bar X)-y)+y+\rge \cA^*\\
& = & \para \p \|\cdot\|_*(\bar X)+\rge \cA^*.
\end{eqnarray*}
All in all, (I) implies (II).

Conversely, if (II),  starting from $\para \p \|\cdot\|_*(\bar X)+\rge \cA^*$, the latter equations shows   $\R^{n\times p}=\aff (\p \|\cdot\|_*(\bar X)+\rge \cA^*)$, while $\ri (\p \|\cdot\|_*(\bar X))\cap \rge \cA^*\neq\emptyset$ implies
\[
0\in\ri (\p \|\cdot\|_*(\bar X))+\rge\cA^*= \ri(\p \|\cdot\|_*(\bar X)+\rge\cA^*)=\inter(\p \|\cdot\|_*(\bar X)+\rge \cA^*),
\]
where the first identity is, again, due to the sum rule for the relative interior, while  the  last identity uses the fact that the relative interior is an interior if (and only if) the parallel subspace (which is here equal to the affine hull) of the convex set in question is the whole space.
\end{proof}

\begin{corollary} \label{cor:Suff} Let  $\bar X$ be  a solution of \eqref{eq:BP} such that \Cref{ass:Suff} holds at $\bar X$. Then $\bar X$ is the unique solution of \eqref{eq:BP}
\end{corollary}

\begin{proof} Combine \Cref{prop:CXV} and \Cref{prop:Suff}. 
\end{proof}

\subsection{More insight} Combining \Cref{cor:Suff}  and \Cref{cor:Equiv}, it follows readily that \Cref{ass:Suff} at $\bar X$ implies that $W(\bar X)\cap \ker\cA=\{0\}$. On the other hand, this argument is not very illuminating when trying to understand the exact interplay of these two types of conditions. Moreover, it is not clear whether \Cref{ass:Suff} might also be necessary for uniqueness of solutions (as it is for its $\ell_1$-analog). We shed some light on these issues now and start with an auxiliary result.

\begin{lemma}\label{lem:Space} Let $\bar X$ satisfy \Cref{ass:Suff}. Then 
\[
\R^{n\times p}=\rge \cA^*+\R_{+}\p\|\cdot\|_*(\bar X).
\]
\end{lemma}
\begin{proof} Set $\cS:=\p\|\cdot\|_*(\bar X)$, and let $Y\in \rge \cA^*\cap \ri \cS$ which exists by \Cref{ass:Suff} (i). Then
\begin{eqnarray*}
\R^{n\times p} & = & \rge \cA^*+\para (\p\|\cdot\|_*(\bar X))\\
& = &  \rge \cA^*+\R_+(\cS-Y)\\
& = & \R_+\left(\rge \cA^*+\cS-Y\right)\\
&= &  \R_+(\rge\cA^*+\cS)\\
& = & \rge\cA^*+\R_+\cS.
\end{eqnarray*}
Here the second identity uses the property  of relative interior points from \eqref{eq:RelInt}. 
\end{proof}

\noindent
As alluded to above, the following result is clear from our previous analysis. We give an explicit proof in the hopes of consolidating the different flavors of the  conditions in  \Cref{ass:Suff} and \Cref{cor:Equiv}, respectively.

\begin{proposition}\label{prop:Bridge} Let $\bar X\in \R^{n \times p}$ with $r\coloneqq\rank \bar X$ and singular value decomposition $\bar X=\bar U\diag(\sigma(\bar X))\bar V^T$. If  \Cref{ass:Suff} holds, then 
$
\ker\cA\cap W(\bar X)=\{0\}.
$
\end{proposition}
\begin{proof} Let $X\in \ker\cA\cap W(\bar X)$. Then, by definition of $W(\bar X)$, there exists $M=\left(\begin{smallmatrix} A-D& B\\ B^T & C\end{smallmatrix}\right) $ with $A\in \bS^r_{+}, C\in \bS^{n-r}_+$, $D\in \bS^{r}_{++}$, $\tr(A)+\tr (C)=\tr(D)$, and $R\in \mathcal{V}_{n-r,p-r}$ such that 
\[
X=\bar UM\left(\begin{smallmatrix} I_r& 0\\ 0 & R\end{smallmatrix}\right)\bar V^T.
\]
On the other hand, by \Cref{lem:Space} and \Cref{prop:SD1} we find $Z\in \rge \cA^*$, $F\in \bB_{op}$ and $t\geq 0$ such that 
\[
X=Z+t\cdot\bar U \left(\begin{smallmatrix} I_r& 0\\ 0 & F\end{smallmatrix}\right)\bar V^T.
\]
Consequently, we have 
\begin{eqnarray*}
\|X\|^2 & = &  \ip{\bar UM\left(\begin{smallmatrix} I_r& 0\\ 0 & R\end{smallmatrix}\right)\bar V^T}{Z+t\cdot\bar U \left(\begin{smallmatrix} I_r& 0\\ 0 & F\end{smallmatrix}\right)\bar V^T}\\
&= &t\cdot \ip{\bar UM\left(\begin{smallmatrix} I_r& 0\\ 0 & R\end{smallmatrix}\right)\bar V^T}{\bar U \left(\begin{smallmatrix} I_r& 0\\ 0 & F\end{smallmatrix}\right)\bar V^T}\\
& = &t\cdot \ip{M\left(\begin{smallmatrix} I_r& 0\\ 0 & R\end{smallmatrix}\right)}{\left(\begin{smallmatrix} I_r& 0\\ 0 & F\end{smallmatrix}\right)}\\
&= &t\cdot \tr\left(\left(\begin{smallmatrix} I_r& 0\\ 0 & RF^T\end{smallmatrix}\right)\cdot  \left(\begin{smallmatrix} A-D& B\\ B^T & C\end{smallmatrix}\right)\right)\\
& = & t\cdot \tr\left( \left(\begin{smallmatrix} A-D& B\\ RF^TB^T & RF^TC\end{smallmatrix}\right)\right)\\
&= & t\cdot\left( \tr(A)-\tr(D)+\tr(RF^TC)  \right)\\
& = & t\cdot\left( \tr(RF^TC)-\tr(C) \right)\\
& \leq  & t\cdot \left( \|RF^T\|_{op}\cdot\|C\|_*-\|C\|_{*}\right)\\
& \leq & 0.
\end{eqnarray*}
Here, the second  identity takes into account that $Z\in \rge \cA^*$ while $\bar UM\left(\begin{smallmatrix} I_r& 0\\ 0 & R\end{smallmatrix}\right)\bar V^T=X\in \ker \cA$. The seventh (last) equality uses the fact that $\tr(A)+\tr (C)=\tr(D)$. The first inequality uses the fact that $C$ is positive semidefinite as well as the `H\"older inequality' for the operator and nuclear norm. The last inequality is due to the fact that $\|R\|_{op}=1$,  $\|F\|_{op}\leq 1$ and the submulitiplicativity of the operator norm.

All in all, we find that  $X=0$ which proves the desired result.
\end{proof}

\noindent
The natural question as to whether \Cref{ass:Suff} is also necessary for uniqueness is answered negatively by the following example.

\begin{example}\label{ex:Counter} Set  $\bE:=\R^{2\times 2}\times \R^{2\times 2}$, and  define $\cA:\R^{2\times 2}\to \bE$ by
\[
\cA(X)=[\left(\begin{smallmatrix}1 & 1 \\ 0 & 0 \end{smallmatrix}\right)X,\; P_{\bA^2}(X)],
\]
where $P_{\bA^2}(X)\coloneq \half (X-X^T)$ is the projection onto the $2\times 2$ skew symmetric matrices $\mathbb A^2$. Set $b:=[\left(\begin{smallmatrix}1 & 0 \\ 0 & 0 \end{smallmatrix}\right), \; \left(\begin{smallmatrix}0 & 0 \\ 0 & 0 \end{smallmatrix}\right)]\in \bE$. Equipped with these choices, consider 
\begin{equation}\label{eq:Ex}
\min_{X\in \R^{2\times 2}} \|X\|_* \st \cA(X)=b.
\end{equation}
The following hold:
\begin{itemize}
\item $\ker \cA=\lin \{\left(\begin{smallmatrix}1 & -1 \\ -1 & 1 \end{smallmatrix}\right)\}$. 
\item $\cA^*:\bE\to \R^{2\times 2}, \quad \cA^*(Y,Z)=\left(\begin{smallmatrix}1 & 0 \\ 1 & 0 \end{smallmatrix}\right)Y+P_{\bA^2}(Z)$.
\item $\rge \cA^*=\set{\left(\begin{smallmatrix}t & s \\ t & s \end{smallmatrix}\right)}{t,s\in \R}+\bA^2$.
\end{itemize}

\noindent
Now, set $\bar X:=\left(\begin{smallmatrix}1 & 0 \\ 0 & 0 \end{smallmatrix}\right)$. Then  $\cA(\bar X)=b$, i.e. $\bar X$ is feasible for \eqref{eq:Ex}. Moreover, by \Cref{prop:SD1},  observe that 
\[
\p \|\cdot\|_*(\bar X)= \set{\left(\begin{smallmatrix}1 & 0 \\ 0 & \beta \end{smallmatrix}\right)}{\beta\in [-1,1]}\, \AND \para (\p \|\cdot\|_*(\bar X))= \set{\left(\begin{smallmatrix}0 & 0 \\ 0 & \beta \end{smallmatrix}\right)}{\beta\in \R}.
\]
It is then an easy exercise to find that 
$
\para(\p \|\cdot\|_*(\bar X))+\rge \cA^*=\R^{2\times 2}.
$
Moreover, we observe that 
\begin{eqnarray*}
0\in \p\|\cdot\|_*(\bar X)+\rge \cA^* & \Longleftrightarrow&   \p\|\cdot\|_*(\bar X)\cap \rge \cA^*\neq \emptyset\\
& \Longleftrightarrow& \exists \beta\in [-1,1], t,s,q\in \R: \; \left(\begin{smallmatrix}1 & 0 \\ 0 & \beta \end{smallmatrix}\right)=\left(\begin{smallmatrix}t & s \\ t & s \end{smallmatrix}\right)+\left(\begin{smallmatrix}0 & q \\ -q & 0 \end{smallmatrix}\right).
 \end{eqnarray*}
The latter system has only one solution  $t=q=1, s=\beta=-1$. In particular, we see that $\bar X$ is a minimizer of \eqref{eq:Ex} and that 
\[
\ri (\p \|\cdot\|_*(\bar X))\cap \rge\cA^*=\emptyset.
\]
In particular, the sufficient condition from \Cref{ass:Suff}  for uniqueness is violated at $\bar X$ (while (ii) is satisfied). In turn, realizing that $\rank \bar X=1$, and  consequently 
\[
W(\bar X)=\set{\left(\begin{smallmatrix}a-d & b \\ b & c \end{smallmatrix}\right)\left(\begin{smallmatrix}1 & 0 \\ 0 &R\end{smallmatrix}\right)}{R^2=1, \left(\begin{smallmatrix}a & b \\ b & c\end{smallmatrix}\right)\in \bS^2_+, \tr\left(\begin{smallmatrix}a-d & b \\ b & c\end{smallmatrix}\right)=0, d>0},
\]
we find that
\begin{eqnarray*}
X\in W(\bar X) \cap \ker \cA & \Longrightarrow & X=\left(\begin{smallmatrix}x& -x \\ -x & x\end{smallmatrix}\right)=\left(\begin{smallmatrix}a-d &\pm b \\  b & \pm c \end{smallmatrix}\right),a-d+c=0.\\
& \Longrightarrow & X=0.
\end{eqnarray*}
Therefore, by  \Cref{cor:Equiv}, $\bar X$ is the unique solution of \eqref{eq:Ex}.
\end{example}

\subsection{Other nuclear norm minimization problems}\label{sec:Other}

\noindent
The following general result affords us to carry over uniqueness results from above to other nuclear norm minimization problems involving a linear operator. The proof relies on {\em Fenchel-Rockafellar duality}  \cite{BoL 00, HUL 01, Roc 70, RoW 98}, and the dual correspondence of {\em strict convexity} and {\em essential smoothness} \cite{Roc 70}. 

\begin{proposition}\label{prop:FR} Let $\cA\in \cL(\bE_1,\bE_2)$,   $g:\bE\to\R$  strictly convex, and  $h:\bE_1\to \rp$ closed, proper convex. Then $\cA$ and $h$ are constant on the solution set
$
\cX^*\coloneqq \argmin_{x\in \bE_1} \left\{g(\cA(x))+h(x)\right\}.
$
\end{proposition}
\begin{proof} Clearly, it suffices to prove  that $\cA$ is constant on $\cX^*$. To this end, observe that the dual problem of 
(the primal problem) $
\min_{x\in \bE_1} \left\{g(\cA(x))+h(x)\right\}
$
reads
$
\max_{y\in \bE_2}\{-g^*(-y)-h^*(\cA^*(y))\}.
$
Since $g$ is finite-valued, strong duality holds, and, in particular, for some dual solution $\bar y\in \bE_2$ and any primal solution $x\in \cX^*$ it holds that, in particular, $\cA(x)\in \p g^*(-\bar y)$, cf.,  e.g. \cite[Example 11.41]{RoW 98}. However, since $g$ is strictly convex, $g^*$ is {\em essentially smooth} \cite[Theorem 26.3]{Roc 70} and hence $\cA(x)=\nabla  g^*(-\bar y)$. Since $x\in \cX^*$ was arbitrary, this proves result.
\end{proof}

\begin{corollary}\label{cor:Other}  Let $\cA\in \cL(\R^{n\times p},\bE)$, $b\in \bE$, $\lambda>0$,   $f:\bE\to\R$  strictly convex, and  let  $\bar X$ be a solution of 
\begin{equation}\label{eq:LASSOtype}
\min_{X\in\R^{n\times n}} f(\cA(X)-b)+\lambda\|X\|_*.
\end{equation}
Then $\bar X$ is the unique solution if and only if 
\[
\set{X\in \R^{n\times p}}{\cA(X)=\cA(\bar X),\; \|X\|_*=\|\bar X\|_*}=\{\bar X\}.
\]
(all of which is the case if and only if $W(\bar X)\cap \ker \cA=\{0\}).$
\end{corollary}

\begin{proof}  Let $\cX=\argmin_{\bR^{n\times p}} \{f(\cA(\cdot)-b)+\lambda\|\cdot\|_*\}$ be the solution set of \eqref{eq:LASSOtype}.  Applying \Cref{prop:FR} to $g:=f((\cdot)-b)$ and $h:=\lambda \|\cdot\|_*$ yields that, in fact,  $\cX=\set{X\in \R^{n\times p}}{\cA(X)=\cA(\bar X),\; \|X\|_*=\|\bar X\|_* }$. Therefore, the claim follows.
\end{proof} 


\section{Final remarks}\label{sec:Final}

\noindent
In this paper, starting from a  study of line segments in the nuclear norm sphere, we established necessary and sufficient conditions for  uniqueness of solutions for  minimizing the nuclear norm over an affine manifold. The central linear-algebraic notion in this regard is {\em simultaneous polarizability}, which formalizes the idea of rotating two (square) matrices in the same fashion to render them positive semidefinite. We then gave another set of sufficient conditions based on the convex geometry of the subdifferential (of the nuclear norm) and its interplay with  (the range of) the ambient linear operator.
A duality-based argument enabled us to transfer these findings to a whole class of nuclear norm-regularized optimization problems with strictly convex fidelity term.

As a topic of future research, we intend to build on this analysis to study stability of nuclear norm(-regularized) optimization problems in terms of the right-hand side $b$ and the regularization parameter $\lambda$. In particular, we would like to study Lipschitz properties of the solution function 
\[
(b,\lambda)\mapsto \argmin_{X\in \R^{n\times p}} \left\{\frac{1}{2}\|\cA(X)-b\|^2+\lambda \|X\|_*\right\}.
\]
This study will rely on a suitable representation of the graph of the subdifferential of the nuclear norm.

\end{document}